\documentclass[review]{elsarticle}

\usepackage{lineno,hyperref}
\modulolinenumbers[5]

\journal{Journal of \LaTeX\ Templates}

\usepackage[utf8]{inputenc}  
\usepackage[T1]{fontenc}   
\usepackage[english]{babel}
\usepackage{amsmath}
\usepackage{amsthm}
\usepackage{amsfonts} 
\usepackage{amssymb}
\usepackage{mathtext}
\usepackage{graphicx}
\usepackage{hyperref} 
\usepackage{array}
\usepackage{multirow}
\usepackage{caption}
\usepackage{subcaption}
\usepackage{yhmath}
\usepackage{tikz}

\newtheorem{theo}{Theorem}[section]

\newtheorem{prop}[theo]{Proposition}
\newtheorem{lem}[theo]{Lemma}

{\theoremstyle{definition} \newtheorem{defi}[theo]{Definition}}
{\theoremstyle{definition} \newtheorem{algo}[theo]{Algorithm}}
{\theoremstyle{definition} }
{\theoremstyle{definition} \newtheorem{ex}[theo]{Example}}
{\theoremstyle{definition} }
{\theoremstyle{remark} }
{\theoremstyle{definition} \newtheorem{fact}[theo]{Facts}}

\newcommand{\Sig}{\mathfrak{S}}
\newcommand{\des}{\text{des}}
\newcommand{\DES}{\text{DES}}
\newcommand{\exc}{\text{exc}}
\newcommand{\EXC}{\text{EXC}}\newcommand{\maj}{\text{maj}}
\newcommand{\ides}{\text{ides}}
\newcommand{\inv}{\text{inv}}
\newcommand{\INV}{\text{INV}}

\newcommand{\asc}{\text{asc}}
\newcommand{\ASC}{\text{ASC}}
\newcommand{\amaj}{\text{amaj}}

\renewenvironment{proof}{\noindent{\bf Proof.}}{\qed}

\begin{document}

\begin{frontmatter}

\title{A new bijection relating $q$-Eulerian polynomials}

\author{Ange Bigeni\fnref{myfootnote}}
\address{Institut Camille Jordan\\
Université Claude Bernard Lyon 1\\
43 boulevard du 11 novembre 1918\\
69622 Villeurbanne cedex\\
France 
}
\ead{bigeni@math.univ-lyon1.fr}

\begin{abstract}
On the set of permutations of a finite set, we construct a bijection which maps the 3-vector of statistics $(\maj-\exc,\des,\exc)$ to a 3-vector $(\maj_2,\widetilde{\des_2},\inv_2)$ associated with the $q$-Eulerian polynomials introduced by Shareshian and Wachs in \textit{Chromatic quasisymmetric functions, arXiv:1405.4269(2014).}
\end{abstract}


\end{frontmatter}

\linenumbers

\section*{Notations}

For all pair of integers $(n,m)$ such that $n<m$, the set $\{n,n+1,\hdots,m\}$ is indifferently denoted by $[n,m]$,$]n-1,m]$,$[n,m+1[$ or $]n-1,m+1[$.

The set of positive integers $\{1,2,3,\hdots\}$ is denoted by $\mathbb{N}_{>0}$.

For all integer $n \in \mathbb{N}_{>0}$, we denote by $[n]$ the set $[1,n]$ and by $\Sig_n$ the set of the permutations of $[n]$. By abuse of notation, we assimilate every $\sigma \in \Sig_n$ with the word $\sigma(1) \sigma(2) \hdots \sigma(n)$.

If a set $S = \{n_1,n_2,\hdots,n_k\}$ of integers is such that $n_1 < n_2 < \hdots < n_k$, we sometimes use the notation $S = \{n_1 < n_2 < \hdots < n_k \}$.

\section{Introduction}
\label{sec:intro}
Let $n$ be a positive integer and $\sigma \in \Sig_n$. A \textit{descent} (respectively \textit{exceedance point}) of $\sigma$ is an integer $i \in [n-1]$ such that $\sigma(i) > \sigma(i+1)$ (resp. $\sigma(i) > i$). The set of descents (resp. exceedance points) of $\sigma$ is denoted by $\DES(\sigma)$ (resp. $\EXC(\sigma)$) and its cardinal by $\des(\sigma)$ (resp. $\exc(\sigma)$). The integers $\sigma(i)$ with $i \in \EXC(\sigma)$ are called exceedance values of $\sigma$.

It is due to MacMahon~\cite{macmahon} and Riordan~\cite{riordan} that $$\sum\limits_{\sigma \in \Sig_n} t^{\des(\sigma)} = \sum\limits_{\sigma \in \Sig_n} t^{\exc(\sigma)} = A_n(t)$$ where $A_n(t)$ is the $n$-th Eulerian polynomial~\cite{euler}. A statistic equidistributed with des or exc is said to be \textit{Eulerian}. The statistic ides defined by $\ides(\sigma) = \des(\sigma^{-1})$ obviously is Eulerian.

The \textit{major index} of a permutation $\sigma \in \Sig_n$ is defined as $$\maj(\sigma) = \sum\limits_{i \in \DES(\sigma)} i.$$ It is also due to MacMahon that
$$\sum\limits_{\sigma \in \Sig_n} q^{\maj(\sigma)} = \prod\limits_{i=1}^n \dfrac{1-q^i}{1-q}.$$
A statistic equidistributed with maj is said to be \textit{Mahonian}. Among Mahonian statistics is the statistic inv, defined by $\inv(\sigma) = |\INV(\sigma)|$ where $\INV(\sigma)$ is the set of \textit{inversions} of a permutation $\sigma \in \Sig_n$, \textit{i.e.} the pairs of integers $(i,j) \in [n]^2$ such that $i < j$ and $\sigma(i) > \sigma(j)$.

In \cite{shareshianwachs}, the authors consider analogous versions of the above statistics : let $\sigma \in \Sig_n$, the set of \textit{2-descents} (respectively \textit{2-inversions}) of $\sigma$ is defined as $$\DES_2(\sigma) = \{i \in [n-1], \sigma(i) > \sigma(i+1)+1\}$$ (resp. $$\INV_2(\sigma) = \{1 \leq i < j \leq n, \sigma(i) = \sigma(j)+1\})$$ and its cardinal is denoted by $\des_2(\sigma)$ (resp. $\inv_2(\sigma)$).

It is easy to see that $\inv_2(\sigma) = \ides(\sigma)$.
The \textit{2-major index} of $\sigma$ is defined as $$\maj_2(\sigma) = \sum\limits_{i \in \DES_2(\sigma)} i.$$

By using quasisymmetric function techniques, the authors of \cite{shareshianwachs} proved the equality
\begin{equation} \label{eq:shareshianwachs}
\sum_{\sigma \in \mathfrak{S}_n} x^{\maj_{2}(\sigma)} y^{\inv_{2}(\sigma)} = \sum_{\sigma \in \mathfrak{S}_n} x^{\maj(\sigma)-\exc(\sigma)} y^{\exc(\sigma)}.
\end{equation}

Similarly, by using the same quasisymmetric function method as in \cite{shareshianwachs}, the authors of \cite{hanceli} proved the equality
\begin{equation}
 \label{eq:hanceli}
 \sum\limits_{\sigma \in \Sig_n} x^{\amaj_2(\sigma)} y^{\widetilde{\asc_2}(\sigma)} z^{\ides(\sigma)} = \sum\limits_{\sigma \in \Sig_n} x^{\maj(\sigma)-\exc(\sigma)} y^{\des(\sigma)} z^{\exc(\sigma)}
\end{equation}
where $\asc_2(\sigma)$ is the number of \textit{2-ascents} of a permutation $\sigma \in \Sig_n$, \textit{i.e.} the elements of the set $\ASC_2(\sigma) = \{i \in [n-1], \sigma(i) < \sigma(i+1)+1\}$, which rises the statistic $\amaj_2$ defined by $$\amaj_2(\sigma) = \sum\limits_{i \in \ASC_2(\sigma)} i,$$ and where
$$\widetilde{\asc_2}(\sigma) = \begin{cases} \asc_2(\sigma) &\text{if $\sigma(1) = 1$,} \\ \asc_2(\sigma)+1 &\text{if $\sigma(1) \neq 1$.}  \end{cases}$$

\begin{defi} \label{def:des2tilde}
Let $\sigma \in \Sig_n$. We consider the smallest $2$-descent $d_2(\sigma)$ of $\sigma$ such that $\sigma(i) = i$ for all $i \in [d_2(\sigma)-1]$ (if there is no such $2$-descent, we define $d_2(\sigma)$ as $0$ and $\sigma(0)$ as $n+1$).

Now, let $d_2'(\sigma) > d_2(\sigma)$ be the smallest $2$-descent of $\sigma$ greater than $d_2(\sigma)$ (if there is no such $2$-descent, we define $d_2'(\sigma)$ as $n$).

We define an inductive property $\mathcal{P}(d_2(\sigma))$ by :
\begin{enumerate}
\item $\sigma(d_2(\sigma)) < \sigma(i)$ for all $(i,j) \in \INV_2(\sigma)$ such that $d_2(\sigma) < i < d_2'(\sigma)$;
\item if $(d_2'(\sigma),j) \in \INV_2(\sigma)$ for some $j$, then either $\sigma(d_2(\sigma)) < \sigma(d_2'(\sigma))$, or $d_2'(\sigma)$ has the property $\mathcal{P}(d_2'(\sigma))$ (where the role of $d_2(\sigma)$ is played by $d_2'(\sigma)$ and that of $d_2'(\sigma)$ by $d_2''(\sigma)$ where $d_2''(\sigma) > d_2'(\sigma)$ is the smallest $2$-descent of $\sigma$ greater than $d_2'(\sigma)$, defined as $n$ if there is no such $2$-descent).
\end{enumerate}
This property is well-defined because $(n,j) \not\in \INV_2(\sigma)$ for all $j \in [n]$.

Finally, we define a statistic $\widetilde{\des_2}$ by :
$$\widetilde{\des_2}(\sigma) = \begin{cases} \des_2(\sigma) &\text{if the property $\mathcal{P}(d_2(\sigma))$ is true,} \\ \des_2(\sigma)+1 &\text{otherwise.}  \end{cases}$$
\end{defi}

In the present paper, we prove the following theorem. 

\begin{theo} \label{theo:existsbijection}
There exists a bijection $\varphi : \mathfrak{S}_n \rightarrow \mathfrak{S}_n$ such that
$$(\maj_{2}(\sigma),\widetilde{\des_2}(\sigma),\inv_{2}(\sigma)) = (\maj(\varphi(\sigma))-\exc(\varphi(\sigma)),\des(\varphi(\sigma)),\exc(\varphi(\sigma))).$$
\end{theo}

As a straight corollary of Theorem \ref{theo:existsbijection}, we obtain the equality

\begin{equation}
\label{eq:bigeni}
\sum\limits_{\sigma \in \Sig_n} x^{\maj_2(\sigma)} y^{\widetilde{\des_2}(\sigma)} z^{\inv_2(\sigma)} = \sum\limits_{\sigma \in \Sig_n} x^{\maj(\sigma)-\exc(\sigma)} y^{\des(\sigma)} z^{\exc(\sigma)}
\end{equation}

which implies Equality (\ref{eq:shareshianwachs}).

The rest of this paper is organised as follows. 

In Section \ref{sec:graphic}, we introduce two graphical representations of a given permutation so as to highlight either the statistic $(\maj-\exc,\des,\exc)$ or $(\maj_2,\widetilde{\des_2},\inv_2)$. Practically speaking, the bijection $\varphi$ of Theorem \ref{theo:existsbijection} will be defined by constructing one of the two graphical representations of $\varphi(\sigma)$ for a given permutation $\sigma \in \Sig_n$.

We define $\varphi$ in Section \ref{sec:varphi}. 

In Section \ref{sec:varphim1}, we prove that $\varphi$ is bijective by constructing $\varphi^{-1}$.

\section{Graphical representations}
\label{sec:graphic}

\subsection{Linear graph}

Let $\sigma \in \Sig_n$. The linear graph of $\sigma$ is a graph whose vertices are (from left to right) the integers $\sigma(1), \sigma(2), \hdots, \sigma(n)$ aligned in a row, where every $\sigma(k)$ (for  $k \in \DES_2(\sigma)$) is boxed, and where an arc of circle is drawn from $\sigma(i)$ to $\sigma(j)$ for every $(i,j) \in \INV_2(\sigma)$.

For example, the permutation $\sigma = 34251 \in \Sig_5$ (such that \linebreak[4]$(\maj_2(\sigma),\widetilde{\des_2}(\sigma),\inv_2(\sigma)) =~(6,3,2) $)  has the linear graph depicted in Figure \ref{fig:FIGm1}.

\begin{figure}[!h]
\centering
\includegraphics[width=2.5cm]{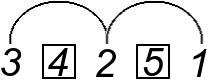}
\caption{Linear graph of $\sigma = 34251 \in \Sig_5$.}
\label{fig:FIGm1}
\end{figure}

\subsection{Planar graph}

Let $\tau \in \Sig_n$. The planar graph of $\tau$ is a graph whose vertices are the integers $1,2,...,n$, organized in ascending and descending slopes (the height of each vertex doesn't matter) such that the $i$-th vertex (from left to right) is the integer $\tau(i)$, and where every vertex $\tau(i)$ with $i \in \EXC(\tau)$ is encircled.

For example, the permutation $\tau = 32541 \in \Sig_5$ (such that \linebreak[4]$(\maj(\tau)-\exc(\tau),\des(\tau),\exc(\tau)) = (6,3,2))$ has the planar graph depicted in Figure \ref{fig:FIG0}.

\begin{figure}[!h]
\centering
\includegraphics[width=2.4cm]{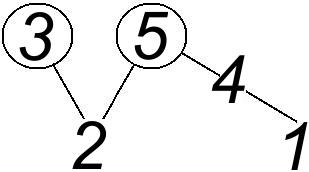}
\caption{Planar graph of $\tau = 32541 \in \Sig_5$.}
\label{fig:FIG0}
\end{figure}

\section{Definition of the map $\varphi$ of Theorem \ref{theo:existsbijection}}
\label{sec:varphi}

Let $\sigma \in \mathfrak{S}_n$. We set $(r,s) = (\des_2(\sigma),\inv_2(\sigma))$, and 
\begin{align*}
\DES_2(\sigma) &= \left \{d_2^k(\sigma), k \in [r]  \right \},\\
\INV_2(\sigma) &= \{(i_l(\sigma),j_l(\sigma)),l \in [s]\}
\end{align*}
with $d_2^k(\sigma) < d_2^{k+1}(\sigma)$ for all $k$ and $i_l(\sigma) < i_{l+1}(\sigma)$ for all $l$.

We intend to define $\varphi(\sigma)$ by constructing its planar graph. To do so, we first construct (in Subsection \ref{subsec:graph}) a graph $\mathcal{G}(\sigma)$ made of $n$ circles or dots organized in ascending or descending slopes such that two consecutive vertices are necessarily in a same descending slope if the first vertex is a circle and the second vertex is a dot. Then, in Subsection \ref{subsec:labelling}, we label the vertices of this graph with the integers $1,2,\hdots,n$ in such a way that, if $y_i$ is the label of the $i$-th vertex $v_i(\sigma)$ (from left to right) of $\mathcal{G}(\sigma)$ for all $i \in [n]$, then :

\begin{enumerate}
\item $y_i <y_{i+1}$ if and only if $v_i$ and $v_{i+1}$ are in a same ascending slope;
\item $y_i > i$ if and only if $v_i$ is a circle.
\end{enumerate}

The permutation $\tau = \varphi(\sigma)$ will then be defined as $y_1 y_2\hdots y_n$, \textit{i.e.} the permutation whose planar graph is the labelled graph $\mathcal{G}(\sigma)$.

With precision, we will obtain 
$$\tau\left(\EXC(\tau)\right)=\{j_k(\sigma),k\in [s]\}$$
(in particular $\exc(\tau) = s = \inv_2(\sigma)$), and 
$$\DES(\tau) = \begin{cases} \{d^k(\sigma),k \in [1,r]\} \text{ if $\widetilde{\des_2}(\sigma) = r$},\\
\{d^k(\sigma),k \in [0,r]\} \text{ if $\widetilde{\des_2}(\sigma) = r+1$}
\end{cases}$$
 for integers $0 \leq d^0(\sigma) < d^1(\sigma) < \hdots < d^{r}(\sigma) \leq n$ (with $d^0(\sigma) = 0 \Leftrightarrow \widetilde{\des_2}(\sigma) = \des_2(\sigma)$) defined by
$$d^k(\sigma) = d_2^k(\sigma) + c_k(\sigma)$$
(with $d_2^0(\sigma) :=0$) where $(c_{k}(\sigma))_{k \in [0,r]}$ is a sequence defined in Subsection \ref{subsec:graph} such that $\sum_k c_k(\sigma) = \inv_2(\sigma) = \exc(\tau)$. Thus, we will obtain $\des(\tau) = \widetilde{\des_2}(\sigma)$ and $\maj(\tau) = \maj_2(\sigma) + \exc(\tau)$.

\subsection{Construction of the unlabelled graph $\mathcal{G}(\sigma)$}
\label{subsec:graph}

We set $\left( d_2^0(\sigma),\sigma(d_2^0(\sigma)) \right) =(0,n+1)$ and $\left( d_2^{r+1}(\sigma),\sigma(n+1) \right) = (n,0)$.

For all $k \in [r]$, we define the top $t_k(\sigma)$ of the 2-descent $d_2^k(\sigma)$ as

\begin{equation} \label{eq:definitionbk} 
t_k(\sigma) = \min\{d_2^l(\sigma),1\leq l \leq k,d_2^l(\sigma) = d_2^k(\sigma) - (k-l)\},
\end{equation}

in other words $t_k(\sigma)$ is the smallest 2-descent $d_2^l(\sigma)$ such that the 2-descents $d_2^l(\sigma),d_2^{l+1}(\sigma),\hdots,d_2^k(\sigma)$ are consecutive integers.

The following algorithm provides a sequence $(c^0_{k}(\sigma))_{k \in [0,r]}$ of nonnegative integers. 

\begin{algo}
\label{algo:suiteck0}
Let $I_r(\sigma) = \INV_2(\sigma)$. For $k$ from $r = \des_2(\sigma)$ down to $0$, we consider the set $S_k(\sigma)$ of sequences $(i_{k_1}(\sigma),i_{k_2}(\sigma), \hdots, i_{k_m}(\sigma))$ such that :
\begin{enumerate} \label{algo}
\item $(i_{k_p}(\sigma),j_{k_p}(\sigma)) \in I_{k}(\sigma)$ for all $p \in [m]$;
\item $t_k(\sigma) \leq i_{k_1}(\sigma) < i_{k_2}(\sigma) < \hdots < i_{k_m}(\sigma)$;
\item $\sigma(i_{k_1}(\sigma)) < \sigma(i_{k_2}(\sigma)) < \hdots < \sigma(i_{k_m}(\sigma))$.
\end{enumerate} 
The \textit{length} of such a sequence is defined as $l = \sum_{p=1}^m n_p$ where $n_p$ is the number of \textit{consecutive} 2-inversions whose beginning is $i_{k_p}$, \textit{i.e.} the maximal number $n_p$ of 2-inversions $(i_{k_p^1}(\sigma),j_{k_p^1}(\sigma)),(i_{k_p^2}(\sigma),j_{k_p^2}(\sigma)),\hdots,(i_{k^{n_p}_p}(\sigma),j_{k^{n_p}_p}(\sigma))$ such that $k^1_p = k_p$ and $j_{k^i_p}(\sigma) = i_{k^{i+1}_p}(\sigma)$ for all $i$.
If $I_k(\sigma) \neq \emptyset$, we consider the sequence $(i_{k^{max}_1}(\sigma),i_{k^{max}_2}(\sigma), \hdots, i_{k^{max}_m}(\sigma)) \in I_k(\sigma)$ whose length $l^{max} = \sum_{p=1}^m n_p^{max}$ is maximal and whose elements $i_{k^{max}_1}(\sigma),i_{k^{max}_2}(\sigma),\hdots,i_{k^{max}_m}(\sigma)$ also are maximal (as integers). Then,
\begin{itemize}
\item if $I_k(\sigma) \neq \emptyset$, we set $c^0_k(\sigma) = l^{max}$ and $$I_{k-1}(\sigma) = I_k(\sigma) \backslash \left( \cup_{p=1}^m \{(i_{k^{max}_i}(\sigma),j_{k^{max}_i}(\sigma)),i\in [n_p^{max}]\} \right);$$
\item else we set $c_k^0(\sigma) = 0$ and $I_{k-1}(\sigma) = I_k(\sigma)$.
\end{itemize}
\end{algo}

\begin{ex}
Consider the permutation $\sigma = 549321867 \in \Sig_9$, with $\DES_2(\sigma) = \{3,7\}$ and $I_2(\sigma) = \INV_2(\sigma) = \{(1,2),(2,4),(3,7),(4,5),(5,6),(7,9)\}$. In Figure \ref{fig:FIG1} are depicted the $\des_2(\sigma)+1 = 3$ steps $k \in \{2,1,0\}$ (at each step, the 2-inversions of the maximal sequence are drawed in red then erased at the following step) : 
\begin{figure}[!h]
\centering
\includegraphics[width=7cm]{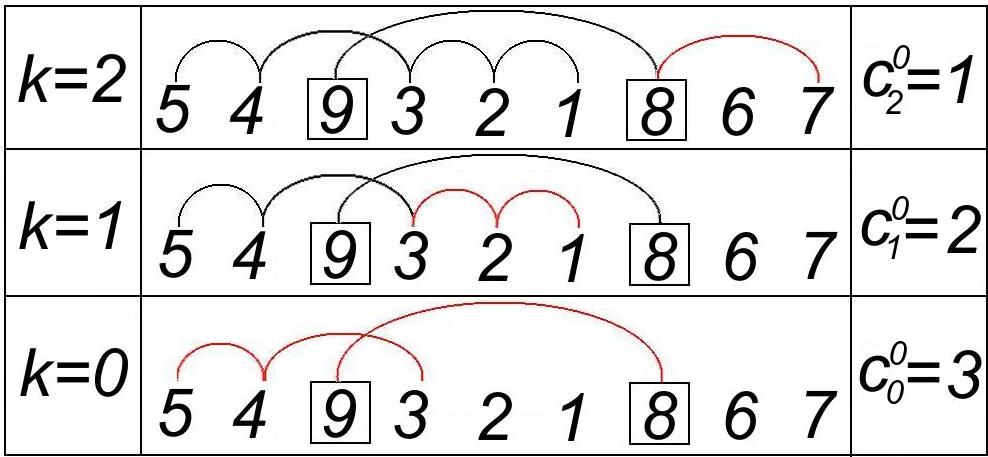}
\caption{Computation of $(c^0_k(\sigma))_{k \in [0,\des_2(\sigma)]}$ for $\sigma = 549321867 \in \Sig_9$.}
\label{fig:FIG1}
\end{figure}
\begin{itemize}
\item $k=2$ : there is only one legit sequence $(i_{k_1}(\sigma)) = (7)$, whose length is $l = n_1 = 1$. We set $c^0_2(\sigma) = 1$ and $I_1(\sigma) = I_2(\sigma) \backslash \{(7,9)\}$.
\item $k=1$ : there are three legit sequences $(i_{k_1}(\sigma)) = (3)$ (whose length is $l = n_1 =1$) then $(i_{k_1}(\sigma)) = (4)$ (whose length is $l=n_1 =2$) and $(i_{k_1}(\sigma)) = (5)$ (whose length is $l =n_1=1$). The maximal sequence is the second one, hence we set $c^0_1(\sigma) = 2$ and $I_0(\sigma) = I_1(\sigma) \backslash \{(4,5),(5,6)\}$.
\item $k=0$ : there are three legit sequences $(i_{k_1}(\sigma),i_{k_2}(\sigma)) = (1,3)$ (whose length is $l = n_1 + n_2 = 2+1=3$) then $(i_{k_1}(\sigma),i_{k_2}(\sigma)) = (2,3)$ (whose length is $l =n_1+n_2 = 1+1=2$) and $(i_{k_1}(\sigma)) = (3)$ (whose length is $l =n_1 = 1$). The maximal sequence is the first one, hence we set $c^0_0(\sigma) = 3$ and $I_{-1}(\sigma) = I_0(\sigma) \backslash \{(1,2),(2,4),(3,7)\} = \emptyset$.
\end{itemize}
\end{ex}

\begin{lem} \label{lem:lemme1}
The sum $\sum_k c^0_k(\sigma)$ equals  $\inv_2(\sigma)$ (\textit{i.e.} $I_{-1}(\sigma) = \emptyset$) and, for all $k \in [0,r] = [0,\des_2(\sigma)]$, we have $c^0_k(\sigma) \leq d_2^{k+1}(\sigma) - d_2^k(\sigma)$ with equality only if $c^0_{k+1}(\sigma) >~0$ (where $c^0_{r+1}(\sigma)$ is defined as $0$).
\end{lem}
\begin{proof}
With precision, we show by induction that, for all $k \in \{\des_2(\sigma),\hdots,1,0\}$, the set $I_{k-1}(\sigma)$ contains no 2-inversion $(i,j)$ such that $d_2^k(\sigma) < i$. For $k=0$, it will mean $I_{-1}(\sigma) = \emptyset$ (recall that $d_2^0(\sigma)$ has been defined as $0$).

$\star$ If $k = \des_2(\sigma) = r$, the goal is to prove that $c^0_r(\sigma) < n - d_2^r(\sigma)$. Suppose there exists a sequence $(i_{k_1}(\sigma),i_{k_2}(\sigma), \hdots, i_{k_m}(\sigma))$ of length $c^0_{r}(\sigma) \geq n-d_2^{r}(\sigma)$ with $t_{r}(\sigma) \leq i_{k_1}(\sigma) < i_{k_2}(\sigma) < \hdots < i_{k_m}(\sigma)$. In particular, there exist $c^0_{r}(\sigma)~\geq~n-d_2^{r}(\sigma)$ 2-inversions $(i,j)$ such that $d_2^{r}(\sigma) < j$, which forces $c^0_{r}(\sigma)$ to equal $n-d_2^{r}(\sigma)$ and every $j > d_2^{r}(\sigma)$ to be the arrival of a 2-inversion $(i,j)$ such that $t_{r}(\sigma) \leq i$. In particular, this is true for $j = d_2^{r}(\sigma)+1$, which is absurd because $\sigma(i) \geq \sigma \left( d_2^{r}(\sigma) \right) > \sigma \left( d_2^{r}(\sigma) + 1 \right) +1$ for all $i \in [t_{r}(\sigma),d_2^{r}(\sigma)]$. Therefore $c^0_{r}(\sigma) < n-d_2^{r}(\sigma)$. Also, it is easy to see that every $i > d_2^{r}(\sigma)$ that is the beginning of a 2-inversion $(i,j)$ necessarily appears in the maximal sequence $\left( i_{k^{max}_1}(\sigma),i_{k^{max}_2}(\sigma), \hdots, i_{k^{max}_m}(\sigma) \right)$ whose length defines $c^0_{r}(\sigma)$, hence $(i,j) \not\in I_{r-1}(\sigma)$.

$\star$ Now, suppose that $c^{0}_{k}(\sigma) \leq d_2^{k+1}(\sigma) - d_2^k(\sigma)$ for some $k \in [\des_2(\sigma)]$ with equality only if $c^0_{k+1}(\sigma) > 0$, and that no 2-inversion $(i,j)$ with $d_2^{k}(\sigma) < i$ belongs to $I_{k-1}(\sigma)$.

If $t_{k-1}(\sigma) = t_k(\sigma)$ (\textit{i.e.}, if $d_2^{k-1}(\sigma) = d_2^k(\sigma) - 1$), since $I_{k-1}(\sigma)$ does not contain any 2-inversion $(i,j)$ with $d_2^{k}(\sigma) <i$, then $c^0_{k-1}(\sigma) \leq 1 = d_2^k(\sigma) - d_2^{k-1}(\sigma)$. Moreover, if $c^0_{k-1}(\sigma) = 1$, then there exists a 2-inversion $(i,j) \in I_{k-1}(\sigma) \subset I_k(\sigma)$ such that $i \in [t_{k-1}(\sigma),d_2^k(\sigma)]$. Consequently $(i)$ was a legit sequence for the computation of $c^0_k(\sigma)$ at the previous step (because $t_k(\sigma) = t_{k-1}(\sigma)$), which implies $c^0_k(\sigma)$ equals at least the length of $(i)$. In particular $c^0_k(\sigma) > 0$.

Else, consider a sequence $(i_{k_1}(\sigma),i_{k_2}(\sigma), \hdots, i_{k_m}(\sigma))$ that fits the three conditions of Algorithm \ref{algo:suiteck0} at the step $k-1$. In particular $t_{k-1}(\sigma) \leq i_{k_1}(\sigma)$. Also $i_{k_m}(\sigma) \leq d_2^k(\sigma)$ by hypothesis. Since $\sigma(i_{k_p}(\sigma)) < \sigma(i_{k_{p+1}}(\sigma))$ for all $p$, and since $\sigma(t_{k-1}(\sigma)) > \sigma(t_{k-1}(\sigma)+1) > \hdots > \sigma \left( d_2^{k-1}(\sigma) \right) > \sigma \left( d_2^{k-1}(\sigma)+1 \right)$, then only one element of the set $[t_{k-1}(\sigma),d_2^{k-1}(\sigma)+1]$ may equal $i_{k_p}(\sigma)$ for some $p \in [m]$. Thus, the length $l$ of the sequence verifies $l \leq d_2^k(\sigma)-d_2^{k-1}(\sigma)$, with equality only if $i_{k_m}(\sigma) = d_2^k(\sigma)$ (which implies $c^0_k(\sigma) > 0$ as in the previous paragraph). In particular, this is true for $l = c^0_{k-1}(\sigma)$.

Finally, as for $k = \des_2(\sigma)$, every $i \in [d_2^{k-1}(\sigma)+1,d_2^k(\sigma)]$ that is the beginning of a 2-inversion $(i,j)$ necessarily appears in the maximal sequence $\left( i_{k^{max}_1}(\sigma),i_{k^{max}_2}(\sigma), \hdots, i_{k^{max}_m}(\sigma) \right)$ whose length defines $c^0_{k-1}(\sigma)$, hence $(i,j) \not\in I_{k-2}(\sigma)$.

So the lemma is true by induction.
\end{proof}

\begin{defi}
\label{def:gzerosigma}
We define a graph $\mathcal{G}^0(\sigma)$ made of circles and dots organised in ascending or descending slopes, by plotting :
\begin{itemize}
\item for all $k \in [0,r]$, an ascending slope of $c_k^0(\sigma)$ circles such that the first circle has abscissa $d_2^k(\sigma) + 1$ and the last circle has abscissa $d_2^k(\sigma) + c_k^0(\sigma)$ (if $c_k^0(\sigma) = 0$, we plot nothing). All the abscissas are distinct because
$$d_2^0(\sigma) + c_0 < d_2^1(\sigma) + c_1 < \hdots < d_2^r(\sigma) + c_r$$ in view of Lemma \ref{lem:lemme1};
\item dots at the remaining $n-s = n-\inv_2(\sigma)$ abscissas from $1$ to $n$, in ascending and descending slopes with respect to the descents and ascents of the word $\omega(\sigma)$ defined by
\begin{equation}
\label{eq:defomega}
\omega(\sigma) = \sigma(u_1(\sigma)) \sigma(u_2(\sigma)) \hdots \sigma(u_{n-s}(\sigma))
\end{equation} 
where $$\{u_1(\sigma) < u_2(\sigma) < \hdots < u_{n-s}(\sigma) \} := \Sig_n \backslash \{i_1(\sigma) < i_2(\sigma) < \hdots < i_{s}(\sigma)\}.$$
\end{itemize}
\end{defi}

\begin{ex}
The permutation $\sigma_0 = 425736981 \in \Sig_9$ (with $\DES_2(\sigma_0) = \{1,4,8\}$ and $\INV_2(\sigma_0) = \{(1,5),(2,9),(4,6),(7,8)\}$), which yields the sequence $(c^0_k(\sigma_0))_{k \in [0,3]} = (1,1,2,0)$ (see Figure \ref{fig:FIG2} where all the 2-inversions involved in the computation of a same $c^0_k(\sigma_0)$ are drawed in a same color) and the word $\omega(\sigma_0) = 53681$, provides the unlabelled graph $\mathcal{G}^0(\sigma_0)$ depicted in Figure \ref{fig:FIG4}.

\begin{figure}[!h]
\centering
\includegraphics[width=7cm]{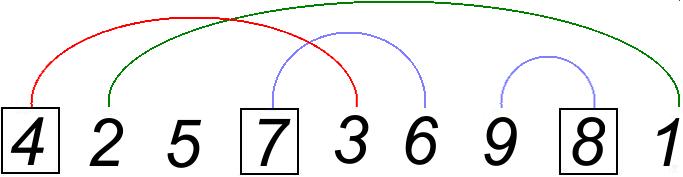}
\caption{$(c^0_k(\sigma_0))_{k \in [0,3]} = (\textcolor{red}{1},\textcolor{green}{1},\textcolor{blue}{2},0).$}
\label{fig:FIG2}
\end{figure}

\begin{figure}[!h]
  \centering
\includegraphics[width=4cm]{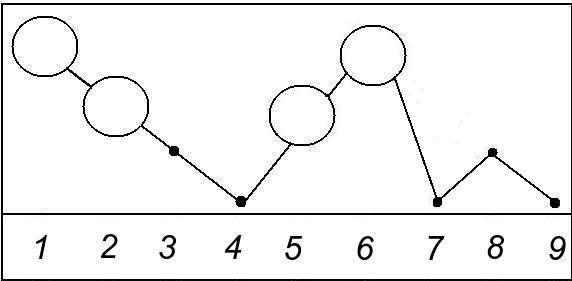}
\caption{Graph $\mathcal{G}^0(\sigma_0)$.}
\label{fig:FIG4}
\end{figure}
\end{ex}

The following lemma is easy.

\begin{lem}
\label{lem:modifsequence} For all $i \in [n]$, if the $i$-th vertex (from left to right) $v_i^0(\sigma)$ of $\mathcal{G}^0(\sigma)$ is a dot and if $i$ is a \textit{descent} of $\mathcal{G}^0(\sigma)$ (i.e., if $v_i^{0}(\sigma)$ and $v_{i+1}^{0}(\sigma)$ are two dots in a same descending slope) whereas $i \not\in \DES_2(\sigma)$, let $k_i$ such that 
$$d_2^{k_i}(\sigma)~+~c^{0}_{k_i}(\sigma)~<~i <~d_2^{k_i+1}(\sigma)$$
and let $p \in [n-s]$ such that $v_i^0(\sigma)$ is the $p$-th dot (from left to right) of $\mathcal{G}^{0}(\sigma)$. Then :
\begin{enumerate}
\item $u_p(\sigma)$ is the greatest integer $u < d_2^{k_i+1}(\sigma)$ that is not the beginning of a 2-inversion of $\sigma$;
\item $u_{p+1}(\sigma)$ is the smallest integer $u > d_2^{k_i+1}(\sigma)$ that is not a 2-descent or the beginning of a 2-inversion of $\sigma$;
\item $c^{0}_k(\sigma) >0$ for all $k$ such that $d_2^{k_i+1}(\sigma) \leq d_2^k(\sigma) \leq u_{p+1}(\sigma)$.
\end{enumerate}
In particular $c^{0}_{k_i+1}(\sigma) > 0$.
\end{lem}

Lemma \ref{lem:modifsequence} motivates the following definition.

\begin{defi}
\label{def:ckgsigma}
For $i$ from $1$ to $n-1$, let $k_i \in [0,r]$ such that
$$d_2^{k_i}(\sigma)~+~c^{0}_{k_i}(\sigma)~<~i <~d_2^{k_i+1}(\sigma).$$
If $i$ fits the conditions of Lemma \ref{lem:modifsequence}, then we define a sequence $(c^i_k(\sigma))_{k \in [0,r]}$ by
\begin{align*}
c^i_{k_i}(\sigma) &= c^{i-1}_{k_i}(\sigma)+1,\\ c^i_{k_i+1}(\sigma) &= c^{i-1}_{k_i+1}(\sigma)-1,\\ c^i_k(\sigma) &= c^{i-1}_k(\sigma)  \text{ for all $k \not\in \{k_i,k_i+1\}$.}
\end{align*}
Else, we define $(c^i_k(\sigma))_{k \in [0,r]}$ as $(c^{i-1}_k(\sigma))_{k \in [0,r]}$.

The final sequence $(c^n_k(\sigma))_{k \in [0,r]}$ is denoted by
$$(c_k(\sigma))_{k \in [0,r]}.$$
\end{defi}

By construction, and from Lemma \ref{lem:lemme1}, the sequence $(c_k(\sigma))_{k \in [0,r]}$ has the same properties as $(c^0_k(\sigma))_{k \in [0,r]}$ detailed in Lemma \ref{lem:lemme1}.

Consequently, we may define an unlabelled graph $$\mathcal{G}(\sigma)$$
by replacing $(c^0_k(\sigma))_{k \in [0,r]}$ with $(c_k(\sigma))_{k \in [0,r]}$ in Definition \ref{def:gzerosigma}.

\begin{ex}
In the graph $\mathcal{G}^0(\sigma)$ depicted in Figure \ref{fig:FIG4} where $\sigma_0 = 425736981 \in \Sig_9$, we can see that the dot $v_3^0(\sigma_0)$ is a descent whereas $3 \not\in \DES_2(\sigma_0)$, hence, from the sequence $(c^0_k(\sigma_0))_{k \in [0,3]} = (1,1,2,0)$, we compute $(c_k(\sigma_0))_{k \in [0,3]} = (1,2,1,0)$ and we obtain the graph 
$\mathcal{G}(\sigma_0)$
depicted in Figure \ref{fig:FIG5}.
\begin{figure}[!h]
  \centering
\includegraphics[width=4cm]{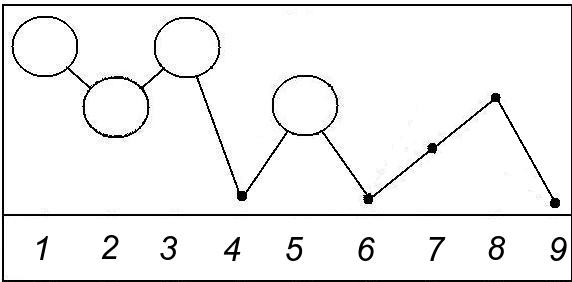}
\caption{Graph $\mathcal{G}(\sigma_0)$.}
\label{fig:FIG5}
\end{figure}
\end{ex}

Let $v_1(\sigma),v_2(\sigma),\hdots,v_n(\sigma)$ be the $n$ vertices of $\mathcal{G}(\sigma)$ from left to right.

By construction, the descents of the unlabelled graph $\mathcal{G}(\sigma)$ (\textit{i.e.}, the integers $i \in [n-1]$ such that $v_i(\sigma)$ and $v_{i+1}(\sigma)$ are in a same descending slope) are the integers
$$d^k(\sigma) = d_2^k(\sigma) + c_k(\sigma)$$
for all $k \in [0,r]$.

\subsection{Labelling of the graph $\mathcal{G}(\sigma)$}
\label{subsec:labelling}

\subsubsection{Labelling of the circles}
We intend to label the circles of $\mathcal{G}(\sigma)$ with the integers $$j_1(\sigma),j_2(\sigma),\hdots,j_{s}(\sigma).$$

\begin{algo}
\label{algo:labelcircles}
For all $i \in [n]$, if the vertex $v_i(\sigma)$ is a circle (hence $i < n$), we label it first with the set 
$$[i+1,n] \cap \{j_1(\sigma),j_2(\sigma),\hdots,j_{s}(\sigma)\}.$$ 

Afterwards, if a circle $v_i(\sigma)$ is found in a descending slope such that there exists a quantity of $a$ circles above $v_i(\sigma)$, and in an ascending slope such that there exists a quantity of $b$ circles above $v_i(\sigma)$, then we remove the $a+b$ greatest integers from the current label of $v_i(\sigma)$ (this set necessarily had at least $a+b+1$ elements) and the smallest integer from every of the $a+b$ labels of the $a+b$ circles above $v_i(\sigma)$ in the two related slopes. At the end of this step, if an integer $j_k(\sigma)$ appears in only one label of a circle $v_i(\sigma)$, then we replace the label of $v_i(\sigma)$ with $j_k(\sigma)$.

Finally, we replace every label that is still a set by the unique integer it may contain with respect to the order of the elements in the sequence $$(j_1(\sigma),j_2(\sigma), \hdots,j_{s}(\sigma))$$ (from left to right).
\end{algo}

\begin{ex}
For $\sigma_0 = 425736981$ (see Figure \ref{fig:FIG2}) whose graph $\mathcal{G}(\sigma_0)$ is depicted in Figure \ref{fig:FIG5}, we have $s = \inv_2(\sigma) = 4$ and $\{j_1(\sigma_0),j_2(\sigma_0),j_3(\sigma_0),j_4(\sigma_0)\} = \{5,6,8,9\}$, which provides first the graph labelled by sets depicted in Figure \ref{fig:FIG6}. Afterwards, since the circle $v_2(\sigma_0)$ is in a descending slope with $a=1$ circle above it (the vertex $v_1(\sigma_0)$) and in an ascending slope with also $b=1$ circle above it (the vertex $v_3(\sigma_0)$), then we remove the $a+b=2$ integers $8$ and $9$ from its label, which becomes $\{5,6\}$, and we remove $5$ from the labels of $v_1(\sigma_0)$ and $v_3(\sigma_0)$. Also, since the label of $v_2(\sigma_0)$ is the only set that contains $5$, then we label $v_2(\sigma_0)$ with $5$ (see Figure \ref{fig:FIG7}). Finally, the sequence $(j_1(\sigma_0),j_2(\sigma_0),j_3(\sigma_0),j_4(\sigma_0)) = (5,9,6,8)$ gives the order (from left to right) of apparition of the remaining integers $6,8,9$ (see Figure \ref{fig:FIG8}).

\begin{figure}[!h]
\begin{minipage}{.32\textwidth}
\centering
\includegraphics[width=4cm]{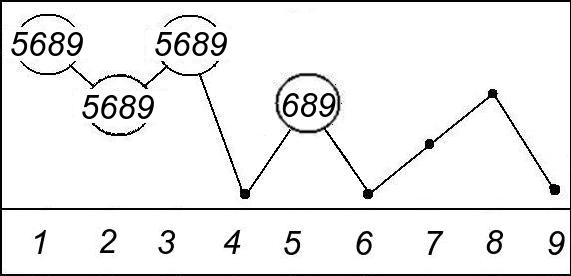}
\caption{}
\label{fig:FIG6}
\end{minipage}
\begin{minipage}{.33\textwidth}
  \centering
\includegraphics[width=4cm]{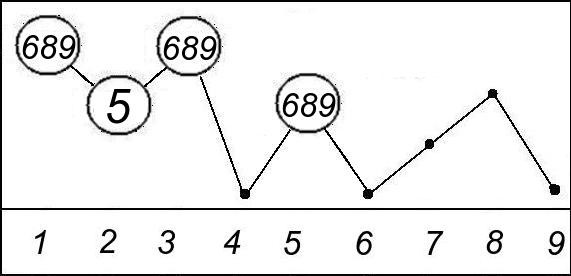}
\caption{}
\label{fig:FIG7}
\end{minipage}%
\begin{minipage}{.33\textwidth}
  \centering
\includegraphics[width=4cm]{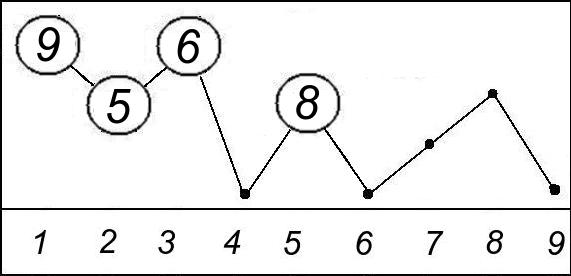}
\caption{}
\label{fig:FIG8}
\end{minipage}
\end{figure}
\end{ex}

\subsubsection{Labelling of the dots}

Let
$$\{p_1(\sigma) < p_2(\sigma) < \hdots < p_{n-s}(\sigma)\} = [n] \backslash \bigsqcup_{k=0}^r ]d_2^k(\sigma),d^k(\sigma)].$$

We intend to label the dots $\{v_{p_i(\sigma)}(\sigma), i \in [n-s]\}$ of $\mathcal{G}(\sigma)$ with the elements of 
$$\{1 = e_1(\sigma) < e_2(\sigma) < \hdots < e_{n-s}(\sigma) \} = [n] \backslash \{j_1(\sigma),j_2(\sigma),\hdots,j_{s}(\sigma)\}.$$

\begin{algo}
\label{algo:labeldots}
\begin{enumerate}
\item For all $k \in [n-s]$, we label first the dot $v_{p_k(\sigma)}(\sigma)$ with the set $$[\min(p_k(\sigma),u_k(\sigma))] \cap ([n] \backslash \{j_1(\sigma),j_2(\sigma),\hdots,j_{s}(\sigma)\})$$ where $u_1(\sigma),u_2(\sigma),\hdots,u_{n-s}(\sigma)$ are the integers introduced in $(\ref{eq:defomega})$. 
\item Afterwards, similarly as for the labelling of the circles, if a dot $v_i(\sigma)$ is found in a descending slope such that $a$ dots are above $v_i(\sigma)$, and in an ascending slope such that $b$ dots are above $v_i(\sigma)$, then we remove the $a+b$ greatest integers from the current label of $v_i(\sigma)$ and the smallest integer from every of the $a+b$ labels of the dots above $v_i(\sigma)$ in the two related slopes. At the end of this step, if an integer $l$ appears in only one label of a dot $v_i(\sigma)$, then we replace the label of $v_i(\sigma)$ with $l$.
\item Finally, for $k$ from $1$ to $n-s$, let 
\begin{equation}
\label{eq:defwik}
w_1^k(\sigma) < w_2^k(\sigma) < \hdots < w_{q_k(\sigma)}^k(\sigma)
\end{equation}
such that $$\{p_{w_i^k(\sigma)}(\sigma),i\} = \left \{p_i(\sigma), \text{ $e_k(\sigma)$ appear in the label of $p_i(\sigma)$}\right \},$$ 
and let $i(k) \in [q_k(\sigma)]$ such that
$$\sigma\left(u_{w_{i(k)}^k(\sigma)}(\sigma)\right) = \min \{ \sigma\left(u_{w_i^k(\sigma)}(\sigma)\right),i \in [q_k(\sigma)]\}.$$
Then, we replace the label of the dot $p_{w_{i(k)}^k(\sigma)}(\sigma)$ with the integer $e_k(\sigma)$ and we erase $e_k(\sigma)$ from any other label (and if an integer $l$ appears in only one label of a dot $v_i(\sigma)$, then we replace the label of $v_i(\sigma)$ with $l$).
\end{enumerate}
\end{algo}

\begin{ex}
For $\sigma_0 = 425736981$ whose graph $\mathcal{G}(\sigma_0)$ has its circles labelled in Figure \ref{fig:FIG8}, the sequence $(u_1(\sigma_0),u_2(\sigma_0),u_3(\sigma_0),u_4(\sigma_0),u_5(\sigma_0)) = (3,5,6,8,9)$ provides first the graph labelled by sets depicted in Figure \ref{fig:FIG9}.
\begin{figure}[!h]
\begin{minipage}{.5\textwidth}
\centering
\includegraphics[width=5cm]{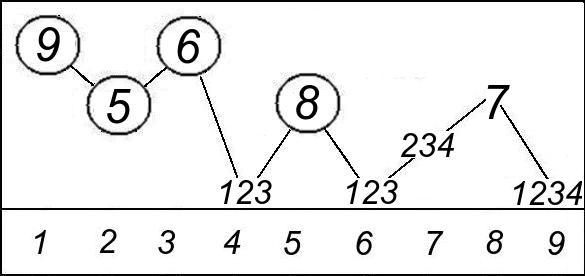}
\caption{}
\label{fig:FIG9}
\end{minipage}
\begin{minipage}{.5\textwidth}
  \centering
\includegraphics[width=5cm]{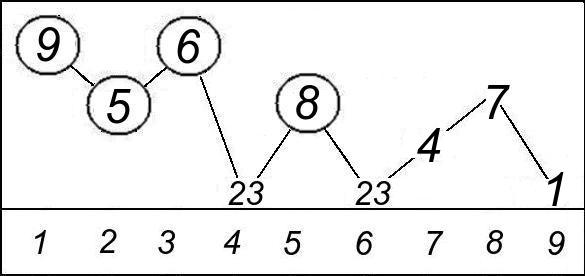}
\caption{}
\label{fig:FIG91}
\end{minipage}%
\end{figure}
\begin{figure}[!h]
\centering
\includegraphics[width=5cm]{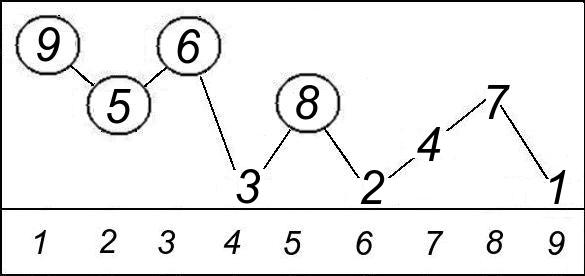}
\caption{Labelled graph $\mathcal{G}(\sigma_0)$.}
\label{fig:FIG10}
\end{figure}
The rest of the algorithm goes from $k=1$ to $n-s=9-4=5$.
\begin{itemize}
\item $k=1$ : in Figure \ref{fig:FIG9}, the integer $e_1(\sigma_0) = 1$ appears in the labels of the dots $v_{p_1(\sigma_0)}(\sigma_0) = v_4(\sigma_0)$, $v_{p_2(\sigma_0)}(\sigma_0) = v_6(\sigma_0)$ and $v_{p_5(\sigma_0)}(\sigma_0) = v_9(\sigma_0)$, so, from $$(\sigma_0(u_1(\sigma_0)),\sigma_0(u_2(\sigma_0)),\sigma_0(u_5(\sigma_0)) = (5,3,1),$$ we label the dot $v_{p_5(\sigma_0)}(\sigma_0) = v_9(\sigma_0)$ with the integer $e_1(\sigma_0) = 1$ and we erase $1$ from any other label, and since the integer $4$ now only appears in the label of the dot $v_7(\sigma_0)$, then we label $v_7(\sigma_0)$ with $4$ (see Figure \ref{fig:FIG91}). 
\item $k=2$: in Figure \ref{fig:FIG91}, the integer $e_2(\sigma_0) = 2$ appears in the labels of the dots $v_{p_1(\sigma_0)}(\sigma_0) = v_4(\sigma_0)$ and $v_{p_2(\sigma_0)}(\sigma_0) = v_6(\sigma_0)$ so, from $$(\sigma_0(u_1(\sigma_0)),\sigma_0(u_2(\sigma_0))) = (5,3),$$ we label the dot $v_{p_2(\sigma_0)}(\sigma_0) = v_6(\sigma_0)$ with the integer $e_2(\sigma_0) = 2$ and we erase $2$ from any other label, which provides the graph labelled by integers depicted in Figure \ref{fig:FIG10}.
\item The three steps $k = 3,4,5$ change nothing because every dot of $\mathcal{G}(\sigma_0)$ is already labelled by an integer at the end of the previous step.
\end{itemize}
So the final version of the labelled graph $\mathcal{G}(\sigma_0)$ is the one depicted in Figure \ref{fig:FIG10}.
\end{ex}

\subsection{Definition of $\varphi(\sigma)$}

By construction of the labelled graph $\mathcal{G}(\sigma)$, the word $y_1 y_2 \hdots y_n$ (where the integer $y_i$ is the label of the vertex $v_i(\sigma)$ for all $i$) obviously is a permutation of the set $[n]$, whose planar graph is $\mathcal{G}(\sigma)$.

We define $\varphi(\sigma) \in \mathfrak{S}_n$ as this permutation.

For the example $\sigma_0 = 425736981 \in \mathfrak{S}_9$ whose labelled graph $\mathcal{G}(\sigma_0)$ is depicted in Figure \ref{fig:FIG10}, we obtain $\varphi(\sigma_0) = 956382471 \in \mathfrak{S}_9$.

In general, by construction of $\tau = \varphi(\sigma) \in \mathfrak{S}_n$, we have
\begin{equation}
\label{eq:exceedancevalues}
\tau \left( \EXC(\tau) \right) =\{j_k(\sigma),k\in [\inv_2(\sigma)]\}
\end{equation}
and
\begin{equation}
\label{eq:descents}
\DES(\tau) = \begin{cases} 
\{d^k(\sigma),k \in [1,\des_2(\sigma)]\} & \text{if $c_0(\sigma) = 0 (\Leftrightarrow d^0(\sigma) =0)$}, \\ \{d^k(\sigma),k \in [0,\des_2(\sigma)]\} & \text{otherwise}. \end{cases}
\end{equation}
Equality (\ref{eq:exceedancevalues}) provides $$\exc(\tau) = \inv_2(\sigma).$$
By $d^k(\sigma) = d_2^k(\sigma) + c_k(\sigma)$ for all $k$, Equality (\ref{eq:descents}) provides $$\maj(\tau) = \maj_2(\sigma) + \sum_{k \geq 0} c_k(\sigma),$$ and by definition of $(c_k(\sigma))_k$ and Lemma \ref{lem:lemme1} we have $\sum_{k \geq 0} c_k(\sigma) = \sum_{k \geq 0} c^0_k(\sigma) = \inv_2(\sigma) = \exc(\tau)$ hence
$$\maj(\tau) - \exc(\tau) = \maj_2(\sigma).$$
Finally, it is easy to see that $\widetilde{\des_2}(\sigma) = \des_2(\sigma)$ if and only if $c_0(\sigma) = 0$, so Equality (\ref{eq:descents}) also provides 
$$\des(\tau) =  \widetilde{\des_2}(\sigma).$$
As a conclusion, we obtain
$$(\maj(\tau)-\exc(\tau),\des(\tau),\exc(\tau)) = (\maj_2(\sigma),\widetilde{\des_2}(\sigma),\inv_2(\sigma))$$
as required by Theorem \ref{theo:existsbijection}.

\section{Construction of $\varphi^{-1}$}
\label{sec:varphim1}

To end the proof of Theorem \ref{theo:existsbijection}, it remains to show that $\varphi : \Sig_n \rightarrow \Sig_n$ is surjective. Let $\tau \in \Sig_n$. We introduce integers $r \geq 0$, $s = \exc(\tau)$, and
$$0 \leq d^{0,\tau} < d^{1,\tau} < \hdots < d^{r,\tau} < n$$
such that
\begin{align*}
\DES(\tau) &= \{d^{k,\tau},k \in [0,r]\} \cap \mathbb{N}_{>0},\\
d^{0,\tau} &= 0 \Leftrightarrow \tau(1) = 1.
\end{align*}
In particular $\des(\tau) = \begin{cases} r &\text{ if $\tau(1) = 1$,}\\
r+1 &\text{ otherwise}. \end{cases}$

For all $k \in [0,r]$, we define
\begin{align*}
c_k^{\tau} &= \EXC(\tau) \cap ]d^{k-1,\tau},d^{k,\tau}] \text{ (with $d^{-1,\tau} := 0$),}\\
d_2^{k,\tau} &= d^{k,\tau} - c_k^{\tau}.
\end{align*}

We have
$$0 = d_2^{0,\tau} < d_2^{1,\tau} < \hdots < d_2^{r,\tau} < n$$
and similarly as Formula \ref{eq:definitionbk}, we define
\begin{equation}
\label{eq:definitionbkbis}
t_k^{\tau} = \min \{d_2^{l,\tau}, 1 \leq l \leq k, d_2^{l,\tau} = d_2^{k,\tau} - (k-l)\}
\end{equation}
for all $k\in [r]$.

We intend to construct a graph $\mathcal{H}(\tau)$ which is the linear graph of permutation $\sigma \in \Sig_n$ such that $\varphi(\sigma) = \tau$.

\subsection{Skeleton of the graph $\mathcal{H}(\tau)$}

We consider a graph $\mathcal{H}(\tau)$ whose vertices $v_1^{\tau},v_2^{\tau},\hdots,v_n^{\tau}$ (from left to right) are $n$ dots, aligned in a row, among which we box the $d_2^{k,\tau}$-th vertex $v_{d_2^{k,\tau}}^{\tau}$ for all $k \in [r]$. We also draw the end of an arc of circle above every vertex $v_j^{\tau}$ such that $j = \tau(i)$ for some $i \in \EXC(\tau)$. 

For the example $\tau_0 = 956382471 \in \Sig_9$ (whose planar graph is depicted in Figure \ref{fig:FIG10}), we have $r = \des(\tau_0)-1 = 3$ and
\begin{align*}
(c_k^{\tau_0})_{k \in [0,3]} &= (1,2,1,0),\\
(d_2^{k,\tau_0})_{k \in [0,3]} &= (1-1,3-2,5-1,8-0) = (0,1,4,8),\\
\tau_0(\EXC(\tau_0)) &= \{5,6,8,9\},
\end{align*}

and we obtain the graph $\mathcal{H}(\tau_0)$ depicted in Figure \ref{fig:FIG11}.
\begin{figure}[!h]
\begin{minipage}{\textwidth}
\centering
\includegraphics[width=5cm]{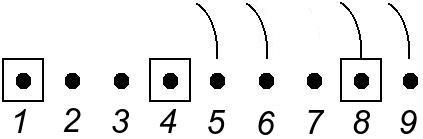}
\caption{Incomplete graph $\mathcal{H}(\tau_0)$.}
\label{fig:FIG11}
\end{minipage}
\end{figure}

In general, by definition of $\varphi(\sigma)$ for all $\sigma \in \Sig_n$, if $\varphi(\sigma) = \tau$, then $r = \des_2(\sigma)$ and $d_2^k(\sigma)$ (respectively $c_k(\sigma),d^k(\sigma),t_k(\sigma)$) equals $d_2^{k,\tau}$ (resp. $c_k^{\tau},d^{k,\tau},t_k^{\tau}$) for all $k \in [0,r]$ and $\{j_l(\sigma),l \in [\inv_2(\sigma)]\} = \tau(\EXC(\tau))$. Consequently, the linear graph of $\sigma$ necessarily have the same skeleton as that of $\mathcal{H}(\tau)$.

The following lemma is easy.

\begin{lem}
\label{lem:3facts}
If $\tau = \varphi(\sigma)$ for some $\sigma \in \Sig_n$, then :
\begin{enumerate}
\label{enum:abc}
\item If $j = \tau(l)$ with $l \in \EXC(\tau)$ such that $l \in ]d_2^{k,\tau},d^{k,\tau}]$, and if $(i,j) \in \INV_2(\sigma)$, then $t_{k}^{\tau} \leq i$.
\item A pair $(i,i+1)$ cannot be a $2$-inversion of $\sigma$ if $i \in \DES_2(\sigma)$ ($\Leftrightarrow$ if the vertex $v_i^{\tau}$ of $\mathcal{H}(\tau)$ is boxed).
\item For all pair $(l,l') \in \EXC(\tau)^2$, if the labels of the two circles $v_l(\sigma)$ and $v_{l'}(\sigma)$ can be exchanged without modifying the skeleton of $\mathcal{G}(\sigma)$, let $i$ and $i'$ such that $(i,l) \in \INV_2(\sigma)$ and $(i',l') \in \INV_2(\sigma)$, then $i < i' \Leftrightarrow l < l'$.
\end{enumerate}
\end{lem}

Consequently, in order to construct the linear graph of a permutation $\sigma \in \Sig_n$ such that $\tau = \varphi(\sigma)$ from $\mathcal{H}(\tau)$, it is necessary to extend the arcs of circles of $\mathcal{H}(\tau)$ to reflect the three facts of Lemma \ref{lem:3facts}. When a vertex is necessarily the beginning of an arc of circle, we draw the beginning of an arc of circle above it. When there is only one vertex $v_i^{\tau}$ that can be the beginning of an arc of circle, we complete the latter by making it start from $v_i^{\tau}$.

\begin{ex}
For $\tau_0 = 956382471 \in \Sig_9$, the graph $\mathcal{H}(\tau_0)$ becomes as depicted in Figure \ref{fig:FIG12}.
\begin{figure}[!h]
\begin{minipage}{\textwidth}
\centering
\includegraphics[width=5cm]{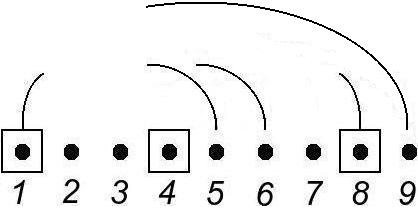}
\caption{Incomplete graph $\mathcal{H}(\tau_0)$.}
\label{fig:FIG12}
\end{minipage}
\end{figure}
Note that the arc of circle ending at $v_6^{\tau_0}$ cannot begin at $v_5^{\tau_0}$ because otherwise, from the third point of Lemma \ref{lem:3facts}, and since $(6,8) = (\tau_0(l),\tau_0(l'))$ with $3 = l < l' = 5$, it would force the arc of circle ending at $v_8^{\tau_0}$ to begin at $v_{i'}^{\tau_0}$ with $6 \leq i'$, which is absurd because a permutation $\sigma \in \Sig_9$ whose linear graph would be of the kind $\mathcal{H}(\tau_0)$ would have $c_2(\sigma) = 2 \neq 1 = c_2^{\tau_0}$. Also, still in view of the third point of Lemma \ref{lem:3facts}, and since $\tau_0^{-1}(9) < \tau^{-1}(6)$, the arc of circle ending at $v_9^{\tau_0}$ must start before the arc of circle ending at $v_6^{\tau_0}$, hence the configuration of $\mathcal{H}(\tau_0)$ in Figure \ref{fig:FIG12}.
\end{ex}

The following two facts are obvious.

\begin{fact}
\label{fact:ascents}
If $\tau = \varphi(\sigma)$ for some $\sigma \in \Sig_n$, then :
\begin{enumerate}
\item A vertex $v_i^{\tau}$ of $\mathcal{H}(\tau)$ is boxed if and only if $i \in \DES_2(\sigma)$. In that case, in particular $i$ is a descent of $\sigma$.
\item If a pair $(i,i+1)$ is not a $2$-descent of $\sigma$ and if $v_i^{\tau}$ is not boxed, then $i$ is an ascent of $\sigma$, \textit{i.e.} $\sigma(i) < \sigma(i+1)$.
\end{enumerate}
\end{fact}

To reflect Facts \ref{fact:ascents}, we draw an ascending arrow (respectively a descending arrow) between the vertices $v_i^{\tau}$ and $v_{i+1}^{\tau}$ of $\mathcal{H}(\tau)$ whenever it is known that $\sigma(i) < \sigma(i+1)$ (resp. $\sigma(i) > \sigma(i+1)$)  for all $\sigma \in \Sig_n$ such that $\varphi(\sigma) = \tau$.

For the example $\tau_0 = 956382471 \in \Sig_9$, the graph $\mathcal{H}(\tau_0)$ becomes as depicted in Figure \ref{fig:FIG13}. Note that it is not known yet if there is an ascending or descending arrow between $v_7^{\tau_0}$ and $v_8^{\tau_0}$.

\begin{figure}[!h]
\begin{minipage}{\textwidth}
\centering
\includegraphics[width=5cm]{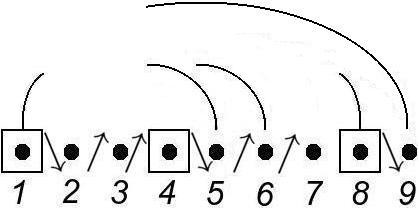}
\caption{Incomplete graph $\mathcal{H}(\tau_0)$.}
\label{fig:FIG13}
\end{minipage}
\end{figure}

\subsection{Completion and labelling of $\mathcal{H}(\tau)$}

The following lemma is analogous to the third point of Lemma \ref{lem:3facts} for the dots instead of the circles and follows straightly from the definition of $\varphi(\sigma)$ for all $\sigma \in \Sig_n$.

\begin{lem}
\label{lem:equivdots} Let $\sigma \in \Sig_n$ such that $\varphi(\sigma) = \tau$.
For all pair $(l,l') \in ([n]\backslash \EXC(\tau))^2$, if the labels of the two dots $v_l(\sigma)$ and $v_{l'}(\sigma)$ can be exchanged without modifying the skeleton of $\mathcal{G}(\sigma)$, let $k$ and $k'$ such that $l = p_k(\sigma)$ and $l' = p_{k'}(\sigma)$, then $\tau(l) < \tau(l') \Leftrightarrow \sigma(u_k(\sigma)) < \sigma(u_{k'}(\sigma))$.
\end{lem}

Now, the ascending and descending arrows between the vertices of $\mathcal{H}(\tau)$ introduced earlier, and Lemma \ref{lem:equivdots}, induce a partial order on the set $\{v_i^{\tau},i \in [n]\}$:

\begin{defi}
\label{def:partialorder}
We define a partial order $\succ$ on $\{v_i^{\tau},i \in [n]\}$ by :
\begin{itemize}
\item $v_i^{\tau } \prec v_{i+1}^{\tau}$ (resp. $v_i^{\tau} \succ v_{i+1}^{\tau}$) if there exists an ascending (resp. descending) arrow between $v_i^{\tau}$ and $v_{i+1}^{\tau}$;
\item $v_i^{\tau} \succ v_j^{\tau}$ (with $i<j$) if there exists an arc of circle from $v_i^{\tau}$ to $v_j^{\tau}$;
\item if two vertices $v_i^{\tau}$ and $v_j^{\tau}$ are known to be respectively the $k$-th and $k'$-th vertices of $\mathcal{H}(\tau)$ that cannot be the beginning of a complete arc of circle, let $l$ and $l'$ be respectively the $k$-th and $k'$-th non-exceedance point of $\tau$ (from left to right), if $(l,l')$ fits the conditions of Lemma \ref{lem:equivdots}, then we set $v_i^{\tau} \prec v_j^{\tau}$ (resp. $v_i^{\tau} \succ v_j^{\tau}$) if $\tau(l) < \tau(l')$ (resp. $\tau(l) > \tau(l')$).
\end{itemize} 
\end{defi}

\begin{ex}
\label{ex:partialorder}
For the example $\tau_0 = 956382471$, according to the first point of Definition \ref{def:partialorder}, the arrows of Figure \ref{fig:FIG13} provide
$$v_1^{\tau_0} \succ v_2^{\tau_0} \prec v_3^{\tau_0} \prec v_4^{\tau_0} \succ v_5^{\tau_0} \prec v_6^{\tau_0} \prec v_7^{\tau_0}$$
and
$$v_8^{\tau_0} \succ v_9^{\tau_0}.$$
\end{ex}

\begin{defi}
\label{def:minimalvertices}
A vertex $v_i^{\tau}$ of $\mathcal{H}(\tau)$ is said to be \textit{minimal} on a subset $S \subset [n]$ if $v_i^{\tau} \not \succ v_j^{\tau}$ for all $j \in S$.
\end{defi}

Let $$1 = e_1^{\tau} < e_2^{\tau} < \hdots < e_{n-s}^{\tau}$$ be the non-exceedance values of $\tau$ (\textit{i.e.}, the labels of the dots of the planar graph of $\tau$).

\begin{algo}
\label{algo:varphim1}
Let $S = [n]$ and $l = 1$. While the vertices $\{v_i^{\tau},i \in [n]\}$ have not all been labelled with the elements of $[n]$, apply the following algorithm.
\begin{enumerate}
\item If there exists a unique minimal vertex $v_i^{\tau}$ of $\tau$ on $S$, we label it with $l$, then we set $l := l+1$ and $S := S \backslash \{v_i^{\tau}\}$. Afterwards,
\begin{enumerate}
\item If $v_i^{\tau}$ is the ending of an arc of circle starting from a vertex $v_j^{\tau}$, then we label $v^{\tau}_{j}$ with the integer $l$ and we set $l := l+1$ and $S := S \backslash \{v_j^{\tau}\}$.
\item If $v^{\tau}_{i}$ is the arrival of an incomplete arc of circle (in particular $i = \tau(l)$ for some $l \in \EXC(\tau)$), we intend to complete the arc by making it start from a vertex $v^{\tau}_{j}$ for some integer $j \in [t_k^{\tau},j[$ (where $l \in ]d_2^{k,\tau},d^{k,\tau}]$) in view of the first point of Lemma \ref{lem:3facts}. We choose $v_j^{\tau}$ as the rightest minimal vertex on $[t_k^{\tau},j[ \cap S$ from which it may start in view of the third point of Lemma \ref{lem:3facts}, and we label this vertex $v^{\tau}_{j}$ with the integer $l$. Then we set $l := l+1$ and $S := S \backslash \{v_j^{\tau}\}$.
\end{enumerate}
Now, if there exists an arc of circle from $v_j^{\tau}$ (for some $j$) to $v_i^{\tau}$, we apply steps (a),(b) and (c) to the vertex $v_j^{\tau}$ in place of $v_i^{\tau}$.
\item Otherwise, let $k \geq 0$ be the number of vertices $v_i^{\tau}$ that have already been labelled and that are not the beginning of arcs of circles. Let
$$l_1 < l_2 < \hdots < l_q$$ be the integers $l \in [n]$ such that $l \geq \tau(l) \geq e_{k+1}^{\tau}$ and such that we can exchange the labels of dots $\tau(l)$ and $e_{k+1}^{\tau}$ in the planar graph of $\tau$ without modifying the skeleton of the graph. It is easy to see that $q$ is precisely the number of minimal vertices of $\tau$ on $S$. Let $l_{i_{k+1}} = \tau^{-1}(e_{k+1}^{\tau})$ and let $v_j^{\tau}$ be the $i_{k+1}$-th minimal vertex (from left to right) on $S$. We label $v_j^{\tau}$ with $l$, then we set $l := l+1$ and $S := S \backslash \{v_j^{\tau}\}$, and we apply steps 1.(a), (b) and (c) to $v_j^{\tau}$ instead of $v_i^{\tau}$.
\end{enumerate}
\end{algo}

By construction, the labelled graph $\mathcal{H}(\tau)$ is the linear graph of a permutation $\sigma \in \Sig_n$ such that
$$\DES_2(\sigma) = \{d_2^{k,\tau},k \in [r]\}$$
and
$$\{j_l(\sigma),l \in [\inv_2(\sigma)]\} = \tau(\EXC(\tau)).$$

\begin{ex}
Consider $\tau_0 = 956382471 \in \Sig_9$ whose unlabelled and incomplete graph $\mathcal{H}(\tau_0)$ is depicted in Figure \ref{fig:FIG13}. 
\begin{itemize}
\item As stated in Example \ref{ex:partialorder}, the minimal vertices of $\tau_0$ on $S = [9]$ are $(v_2^{\tau_0},v_5^{\tau_0},v_9^{\tau_0})$. Following step 2 of Algorithm \ref{algo:varphim1}, $k=0$ and the integers $ l \in [9]$ such that $\tau_0(l) \geq e_{k+1}^{\tau_0} = 1$ and such that the labels of dots $\tau_0(l)$ can be exchanged with $1$ in the planar graph of $\tau_0$ (see Figure \ref{fig:FIG10}) are $(l_1,l_2,l_3) = (4,6,9)$. By $\tau_0^{-1}(1) = 9 = l_3$, we label the third minimal vertex on $[9]$, \textit{i.e.} the vertex $v_9^{\tau_0}$, with the integer $l=1$.

Afterwards, following step 1.(b), since $v_9^{\tau_0}$ is the arrival of an incomplete arc of circle starting from a vertex $v_j^{\tau_0}$ with $1 = t_1^{\tau_0} \leq j$, and with $j <5$ because that arc of circle must begin before the arc of circle ending at $v_6^{\tau_0}$ in view of Fact 3 of Lemma \ref{lem:3facts}, we complete that arc of circle by making it start from the unique minimal vertex $v_j^{\tau_0}$ on $[1,5[$, \textit{i.e.} $j=2$, and we label $v_2^{\tau_0}$ with the integer $l = 2$ (see Figure \ref{fig:FIG14}). Note that as from now we know that the arc of circle ending at $v_5^{\tau_0}$ necessarily begins at $v_1^{\tau_0}$, because otherwise $v_1^{\tau_0}$, being the beginning of an arc of circle, would be the beginning of the arc of circle ending at $v_6^{\tau_0}$, which is absurd in view of Fact 3 of Lemma \ref{lem:3facts} because $\tau_0^{-1}(9) < \tau^{-1}(6)$, so we complete that arc of circle by making it start from $v_1^{\tau_0}$, which has been depicted in Figure \ref{fig:FIG14}.

\begin{figure}[!h]
\begin{minipage}{\textwidth}
\centering
\includegraphics[width=5cm]{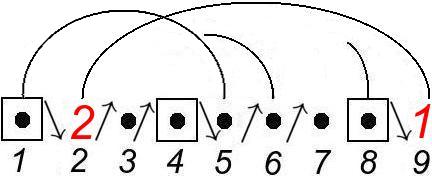}
\caption{Beginning of the labelling of $\mathcal{H}(\tau_0)$.}
\label{fig:FIG14}
\end{minipage}
\end{figure}

We now have $S = [9] \backslash \{2,9\}$ and $l=3$.

\item From Figure \ref{fig:FIG14}, the minimal vertices on $S = [9] \backslash \{2,9\}$ are $(v_3^{\tau_0},v_5^{\tau_0})$. Following step 2 of Algorithm \ref{algo:varphim1}, $k=1$ and the integers $l \in [9]$ such that $l \geq \tau_0(l) \geq e_{k+1}^{\tau_0} = 2$ and such that the labels of dots $\tau_0(l)$ can be exchanged with $2$ in the planar graph of $\tau_0$ (see Figure \ref{fig:FIG10}) are $(l_1,l_2) = (4,6)$. By $\tau_0^{-1}(2) = 6 = l_2$, we label the second minimal vertex on $S$, \textit{i.e.} the vertex $v_5^{\tau_0}$, with the integer $l=3$.

Afterwards, following step 1.(a), since $v_5^{\tau_0}$ is the arrival of the arc of circle starting from the vertex $v_1^{\tau_0}$, we label $v_1^{\tau_0}$ with the integer $l = 4$ (see Figure \ref{fig:FIG15}).

\begin{figure}[!h]
\begin{minipage}{\textwidth}
\centering
\includegraphics[width=5cm]{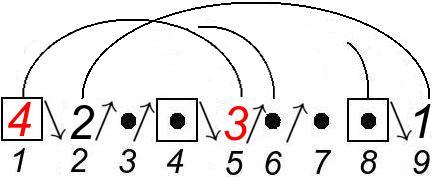}
\caption{Beginning of the labelling of $\mathcal{H}(\tau_0)$.}
\label{fig:FIG15}
\end{minipage}
\end{figure}

We now have $S = [9] \backslash \{1,2,5,9\}$ and $l=5$.

\item From Figure \ref{fig:FIG15}, the minimal vertices on $S = \{3,4,6,7,8\}$ are $(v_3^{\tau_0},v_6^{\tau_0})$. Following step 2 of Algorithm \ref{algo:varphim1}, $k=2$ and the integers $l \in [9]$ such that $l \geq \tau_0(l) \geq e_{k+1}^{\tau_0} = 3$ and such that the labels of dots $\tau_0(l)$ can be exchanged with $3$ in the planar graph of $\tau_0$ (see Figure \ref{fig:FIG10}) are $(l_1,l_2) = (4,7)$. By $\tau_0^{-1}(3) = 4 = l_1$, we label the first minimal vertex on $S$, \textit{i.e.} the vertex $v_3^{\tau_0}$, with the integer $l=5$ (see Figure \ref{fig:FIG16}). Note that as from now we know that the arc of circle ending at $v_6^{\tau_0}$ necessarily begins at $v_4^{\tau_0}$ since it it is the only vertex left it may start from. Consequently, the arc of circle ending at $v_8^{\tau_0}$ necessarily starts from $v_7^{\tau_0}$ (otherwise it would start from $v_6^{\tau_0}$, which is prevented by Definition \ref{def:partialorder} because we cannot have $v_8^{\tau_0} \prec v_6^{\tau_0} \prec v_7^{\tau_0} \prec v_8^{\tau_0}$). The two latter remarks are taken into account in Figure \ref{fig:FIG16}.

\begin{figure}[!h]
\begin{minipage}{\textwidth}
\centering
\includegraphics[width=5cm]{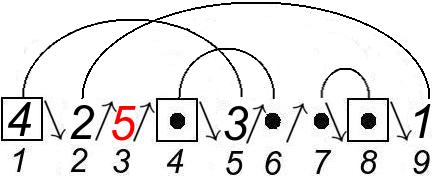}
\caption{Beginning of the labelling of $\mathcal{H}(\tau_0)$.}
\label{fig:FIG16}
\end{minipage}
\end{figure}

We now have $S = \{4,6,7,8\}$ and $l=6$.

\item From Figure \ref{fig:FIG16}, there is only one minimal vertex on $S = \{4,6,7,8\}$, \textit{i.e.} the vertex $v_6^{\tau_0}$. Following step 1 of Algorithm \ref{algo:varphim1}, we label $v_6^{\tau_0}$ with $l=6$.

Afterwards, following step 1.(a), since $v_6^{\tau_0}$ is the arrival of the arc of circle starting from the vertex $v_4^{\tau_0}$, we label $v_4^{\tau_0}$ with the integer $l = 7$ (see Figure \ref{fig:FIG17}).

\begin{figure}[!h]
\begin{minipage}{\textwidth}
\centering
\includegraphics[width=5cm]{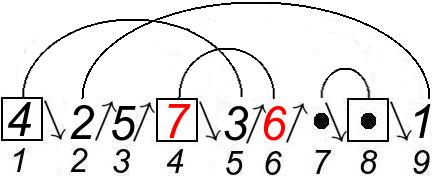}
\caption{Beginning of the labelling of $\mathcal{H}(\tau_0)$.}
\label{fig:FIG17}
\end{minipage}
\end{figure}

We now have $S = \{7,8\}$ and $l=8$.

\item From Figure \ref{fig:FIG17}, there is only one minimal vertex on $S = \{7,8\}$, \textit{i.e.} the vertex $v_6^{\tau_0}$. Following step 1 of Algorithm \ref{algo:varphim1}, we label $v_6^{\tau_0}$ with $l=8$.

Afterwards, following step 1.(a), since $v_8^{\tau_0}$ is the arrival of the arc of circle starting from the vertex $v_7^{\tau_0}$, we label $v_7^{\tau_0}$ with the integer $l = 9$ (see Figure \ref{fig:FIG18}).

\begin{figure}[!h]
\begin{minipage}{\textwidth}
\centering
\includegraphics[width=5cm]{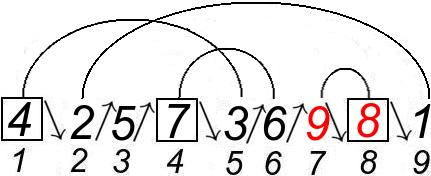}
\caption{Labelled graph $\mathcal{H}(\tau_0)$.}
\label{fig:FIG18}
\end{minipage}
\end{figure}
\end{itemize}

As a conclusion, the graph $\mathcal{H}(\tau_0)$ is the linear graph of the permutation $\sigma_0 = 425736981 \in \Sig_9$, which is mapped to $\tau_0$ by $\varphi$.
\end{ex}

\begin{prop}
\label{prop:bijection}
We have $\varphi(\sigma) = \tau$, hence $\varphi$ is bijective.
\end{prop}

\begin{proof}
By construction , for all $k \in [0,\des_2(\sigma)] = [0,r]$,
\begin{align*}
d_2^k(\sigma) &= d^{k,\tau} - c_k^{\tau},\\
c_k(\sigma) &= c_k^{\tau},\\
d^k(\sigma) &= d_2^k(\sigma) + c_k(\sigma) = d_2^{k,\tau}+c_k^{\tau} = d^{k,\tau},
\end{align*}
so $\mathcal{G}(\sigma)$ has the same skeleton as the planar graph of $\tau$, \textit{i.e.} $\DES(\varphi(\sigma)) = \DES(\tau)$ and $\EXC(\varphi(\sigma)) = \EXC(\tau)$.

The labels of the circles of $\mathcal{G}(\sigma)$ are the elements of
$$\{j_l(\sigma),l \in [s]\} = \tau(\EXC(\tau)),$$
and by construction of $\sigma$, every pair $(l,l') \in \EXC(\tau)^2$ such that we can exchange the labels $\tau(l)$ and $\tau(l')$ in the planar graph of $\tau$ is such that
$$i < i' \Leftrightarrow l < l'$$
where $(i,\tau(l))$ and $(i,\tau(l'))$ are the two corresponding $2$-inversions of $\sigma$. Consequently, by definition of $\varphi(\sigma)$, the labels of the circles of $\mathcal{G}(\sigma)$ appear in the same order as in the planar graph of $\tau$ (\textit{i.e.} $\varphi(\sigma)(i) = \tau(i)$ for all $i \in \EXC(\varphi(\sigma)) = \EXC(\tau)$).

As a consequence, the dots of $\mathcal{G}(\sigma)$ and the planar graph of $\tau$ are labelled by the elements 
$$1 = e_1(\sigma) = e_1^{\tau} < e_2(\sigma) = e_2^{\tau} < \hdots < e_{n-s}(\sigma) = e_{n-s}^{\tau}.$$
As for the labels of the circles, to show that the above integers appear in the same order among the labels of $\mathcal{G}(\sigma)$ and the planar graph of $\tau$, it suffices to prove that 
$$\varphi(\sigma)^{-1}(e_i^{\tau}) < \varphi(\sigma)^{-1}(e_j^{\tau}) \Leftrightarrow \tau^{-1}(e_i^{\tau}) < \tau^{-1}(e_j^{\tau})$$
for all pair $(i,j)$ such that we can exchange the labels $e_i^{\tau}$ and $e_j^{\tau}$ in the planar graph of $\tau$ (hence in $\mathcal{G}(\sigma)$ since the two graphs have the same skeleton). This is guaranteed by Definition \ref{def:partialorder} because the vertices $v_i^{\tau}$ that are not the beginning of an arc of circle correspond with the labels of the dots of the planar graph of $\tau$.

As a conclusion, the planar graph of $\tau$ is in fact $\mathcal{G}(\sigma)$, \textit{i.e.} $\tau = \varphi(\sigma)$.
\end{proof}

\section{Open problem}
In view of Formula (\ref{eq:hanceli}) and Theorem \ref{theo:existsbijection}, it is natural to look for a bijection $\Sig_n \rightarrow \Sig_n$ that maps $(\maj_2,\widetilde{\des_2},\inv_2)$ to $(\amaj_2,\widetilde{\asc_2},\ides)$.

Recall that $\ides = \des_2$ and that for a permutation $\sigma \in \Sig_n$, the equality $\widetilde{\des_2}(\sigma) = \des_2(\sigma)$ is equivalent to $\varphi(\sigma)(1) = 1$, which is similar to the equivalence $\widetilde{\asc_2}(\tau) = \asc_2(\tau) \Leftrightarrow \tau(1) = 1$ for all $\tau \in \Sig_n$.

Note that if $\DES_2(\sigma) = \bigsqcup_{p = 1}^r [i_p,j_p]$ with $j_p+1 < i_{p+1}$ for all $p$, the permutation $\pi = \rho_1 \circ \rho_2 \circ \hdots \circ \rho_r \circ \sigma$, where $\rho_p$ is the $(j_p-i_p+2)$-cycle
$$\begin{pmatrix}
 i_p & i_p+1 & i_p+2 & \hdots & j_p & j_p+1\\
 \sigma(j_p+1) & \sigma(j_p) & \sigma(j_p-1) & \hdots & \sigma(i_p+1) & \sigma(i_p) 
 \end{pmatrix}$$
 for all $p$, is such that $\DES_2(\sigma) \subset \ASC_2(\pi)$ and $\INV_2(\sigma) = \INV_2(\pi)$. One can try to get rid of the eventual unwanted $2$-ascents $i \in \ASC_2(\pi) \backslash \DES_2(\sigma)$ by composing $\pi$ with adequate permutations.
 
 \nocite{*}

\section*{References}
\bibliographystyle{hep}
\bibliography{biblio}
\label{sec:biblio}

\end{document}